\documentclass{amsart}

\usepackage{color}
\usepackage{soul}
\usepackage{amsmath, amssymb}

\usepackage{hyperref}
\hypersetup{
    colorlinks=true, 
    linktoc=all,     
    linkcolor=blue,  
}

\newcommand{\dotieconcat}[2]{
  \text{\raisebox{.8ex}{$\smallfrown$}}%
}

\newcommand{\Q}{\mathbb{Q}}
\newcommand{\C}{\mathbb{C}}

\newcommand{\N}{\mathbb{N}}

\newcommand{\cstar}{$\mathrm{C}^*$}

\newbox\gnBoxA
\newdimen\gnCornerHgt
\setbox\gnBoxA=\hbox{$\ulcorner$}
\global\gnCornerHgt=\ht\gnBoxA
\newdimen\gnArgHgt
\def\Quinequote #1{%
    \setbox\gnBoxA=\hbox{$#1$}%
    \gnArgHgt=\ht\gnBoxA%
    \ifnum     \gnArgHgt<\gnCornerHgt \gnArgHgt=0pt%
    \else \advance \gnArgHgt by -\gnCornerHgt%
    \fi \raise\gnArgHgt\hbox{$\ulcorner$} \box\gnBoxA %
    \raise\gnArgHgt\hbox{$\urcorner$}}

\DeclareMathOperator{\spec}{sp}

\newcommand{\Cr}{C^*_r(G)}
\newcommand{\Qi}{\Q(i)}

\theoremstyle{plain}
\newtheorem{theorem}{Theorem}[section]
\newtheorem{lemma}[theorem]{Lemma}

\newtheorem{question}[theorem]{Question}
\newtheorem{fact}[theorem]{Fact}

\theoremstyle{definition}
\newtheorem{definition}[theorem]{Definition}

\numberwithin{equation}{section}

\begin{document}
\title[Computable presentations of RD group \cstar-algebras]{Computable presentations of reduced group \cstar-algebras of groups with the rapid decay property}

\author{Isaac Goldbring}
\address{University of California, Irvine}
\email{igoldbri@uci.edu}
\urladdr{https://www.math.uci.edu/~isaac}

\begin{abstract}
We show that whenever $G$ is a finitely generated group with the rapid decay property (property RD), then the norm on the reduced group \cstar-algebra $C^*_r(G)$ is computable if (and only if) $G$ has solvable word problem.  We also give an example of a finitely presented group $G$ with solvable word problem for which the norm on $C^*_r(G)$ is not computable.
\end{abstract}
\maketitle

\section{Introduction}\label{sec:Intro}

In \cite{canyou}, Fritz, Netzer, and Thomas began the study of computable presentations of group \cstar-algebras.  One of the results that they showed was that if $G$ is a finitely generated group with solvable word problem, then the norm $\|\cdot\|_r$ on $C^*_r(G)$ is \emph{left c.e.}, meaning that there is an algorithm such that, upon input $a\in C^*_r(G)$, produces a sequence of lower bounds converging to $\|a\|_r$.  (We will be a bit more precisely about what exactly this means in a moment.)  It is an interesting question as to when this conclusion can be upgraded to \emph{computability} of the norm, meaning that there is also an algorithm for producing a sequence of upper bounds converging to $\|a\|_r$.  The authors of \cite{canyou} establish such an algorithm for finitely presented \emph{amenable groups} with solvable word problem; the amenability assumption allows one to use the fact that a computable sequence of converging upper bounds for the universal group \cstar-algebra $C^*(G)$ (which always exists in the finitely presented case) can be used to give computable upper bounds for $C^*_r(G)$ as well.  We note that the assumption of solvable word problem is necessary:  if the norm on $C^*_r(G)$ is computable, it follows easily that $G$ has solvable word problem.  

In the author's recent preprint \cite{goldbring}, he showed that the group \cstar-algebras $C^*_r(G)$ for $G$ a hyperbolic group are computable using a similar idea to the amenable case, this time using an amenable action of the group on its Gromov boundary.

In this note, we extend this latter result to a larger class of groups, namely those with the \emph{rapid decary property} (or \emph{property RD}) (see the next section for a definition):

\begin{theorem}
If $G$ is a finitely generated group with solvable word problem and property RD, then the norm on $C^*_r(G)$ is computable.
\end{theorem}

The motivation for studying these groups in regards to computable presentations of the reduced group \cstar-algebra is obvious:  property RD allows one to bound the the norm in the reduced group \cstar-algebra \emph{from above} by another norm which is obviously computable (when the group has solvable word problem).  The actual implementation inovlves some nontrivial facts about groups with property RD, but the basic idea is that mentioned in the previous sentence.  

Groups with property RD abound and, by results of Jolissaint \cite{jolissaint} and de la Harpe \cite{DLH}, they include the hyperbolic groups in our earlier preprint.  Groups with property RD have been useful in a number of arenas and we refer the reader to the survey \cite{chatterji} for an introduction to the notion, a list of examples (and non-examples), and some sample applications.  Groups with property RD have been of interest lately in connection with the classification programme for nuclear \cstar-algebras, beginning with the recent fundamental result of \cite{strict} showing that $C^*_r(\mathbb{F}_2)$ has the strict comparison property (where property RD for $\mathbb{F}_2$, a result due to Haagerup \cite{haagerup}, played a major role).

It occurred to the author that there seemed to be no known example of a finitely generated group with solvable word problem for which the norm on the reduced group \cstar-algebra is not computable.  In the last section, we offer one such example.  The group is not ``natural'' in any combinatorial or geometric sense but instead comes from a trick encoding the solvability of the halting problem into a direct sum and also involves Kester's characterization of amenability for groups in terms of simple random walks on the group.

We thank Thomas Sinclair for many useful discussions about this work.

\section*{AI Disclosure Statement}

The idea that groups with property RD should have computable reduced group \cstar-algebras is due to the author (who was familiar with these groups in connection with the strict comparison result mentioned above) as is the basic algorithm for computing upper bounds.  However, the author struggled with the completeness argument for the algorithm and consulted Claude Opus 5 for direction.  It was here that Claude pointed the author to the missing ingredient, namely Connes' result about spectral invariance of the algebra of rapidly decreasing functions in the reduced group \cstar-algebra (Fact \ref{fact2} below).  The article was written by the author and he takes full responsibility for the correctness of the arguments. 

\section{Proof of the main theorem}

Throughout, we let $G$ denote a group with a finite symmetric generating set $S$, word length $\ell$, and left regular representation $\lambda$ on $\ell^2(G)$. We identify $\mathbb{C}[G]$ with the set of finitely supported functions $f:G\to \mathbb{C}$.  The reduced group \cstar-algebra of $G$ is $\Cr=\overline{\lambda(\C[G])}^{\|\cdot\|}$. For $x\in \mathbb{C}[G]$, we set $\|x\|_r:=\|\lambda(x)\|$.

To keep the paper short, we will not give the official definition of a presentation of a \cstar-algebra; indeed, since we only consider so-called ``standard'' presentations of reduced group \cstar-algebras, we can work with a more elementary approach.  Consequently, we say that the norm on $C^*_r(G)$ is \emph{computable} if there is an algorithm such that, upon input finitely many $c_1,\ldots,c_n\in \mathbb{Q}(i)$ and finitely many words $w_1,\ldots,w_n\in S^*$ (where, following typical notation in automata theory, $S^*$ denotes the set of words on the alphabet $S$), and a rational $\epsilon>0$, returns an interval $I\subseteq \mathbb{R}$ with $|I|<\epsilon$ for which $\|\sum_{i=1}^n c_i\bar w_i\|_r\in I$; here, $\bar w_i$ denotes the element of $G$ corresponding to $w_i$.  In the sequel, we might be a bit more casual and simply speak of inputting an element of $\mathbb{Q}(i)[G]$ into an algorithm instead of the more formal description given in the previous sentence.

We now begin describing the space of rapidly decreasing functions on a group.

\begin{definition}\label{def:rd}
For $s\ge0$ and $f\in \mathbb{C}[G]$, set
\[
\|f\|_{H^s}:=\left(\sum_{g\in G}|f(g)|^2\big(1+\ell(g)\big)^{2s}\right)^{1/2}.
\]
\end{definition}

The following lemma is clear but important for the proof of the main theorem:

\begin{lemma}\label{sobolevcomp}
If $G$ has solvable word problem, then there is an algorithm such that, upon input $x\in \mathbb{Q}(i)[G]$ and $s\in \N$, computes the rational number $\|x\|_{H^s}^2$.
\end{lemma}

\begin{definition}
We say that $G$ has \emph{property RD} if there exist $C>0$ and $s\ge0$ so that, for all $f\in \mathbb{C}[G]$, we have
\begin{equation}\label{eq:rd}
\|f\|_r\ \le\ C\,\|f\|_{H^s}.
\end{equation}
Replacing $(C,s)$ by $(\lceil C\rceil,\lceil s\rceil)$ preserves \eqref{eq:rd}, so we assume throughout that $C\in\N$ and $s\in\N$.  We call the pair $(C,s)$ a pair of \emph{RD constants} for $G$.
\end{definition}

In general, for $f\in \mathbb{C}[G]$, we have $\|f\|_2\le\|f\|_{H^s}$.  As a result, the inclusion $\C[G]\subseteq\ell^2(G)$ extends to an injection of the $\|\cdot\|_{H^s}$-completion of $G$ into $\ell^2(G)$; this completion is denoted $H^s(G)$ and, under the above identification, we have
\[
H^s(G)=\Big\{f:G\to\C\ \Big|\ \textstyle\sum_{g\in G}|f(g)|^2(1+\ell(g))^{2s}<\infty\Big\}.
\] 

We set $H^\infty(G):=\bigcap_{s\in\N}H^s(G)$, the space of so-called \emph{rapidly decreasing functions on $G$}.  We equip $H^\infty(G)$ with the Fr\'echet topology defined by the family of seminorms $\|\cdot\|_{H^s}$.  If $G$ has property RD, then \eqref{eq:rd} implies that there is a natural (continuous) embedding $H^\infty(G)\hookrightarrow\Cr$ with dense range.  In fact, viewing $C^*_r(G)$ as a subspace of $\ell^2(G)$, we have $H^\infty(G)\subseteq C^*_r(G)$ (literal inclusion of functions on $G$).  (This is in fact an equivalent characterization of having property RD; see \cite[Proposition 4.1]{chatterji}.)

We will need the following two important facts about $H^\infty(G)$ when $G$ has property RD.  The first is essentially due to Jolissaint (see \cite[Lemma 1.2.4]{jolissaint}) although we use a slightly stronger form due to Lafforgue \cite{lafforgue}:

\begin{fact}\label{fact1}
Suppose that $G$ has property RD with RD constants $(C,s)$.  Then $H^s(G)$ is a Banach subalgebra of $C^*_r(G)$.  
\end{fact}

The second fact is due to Connes but appears for print in the first time in a paper of Jolissaint \cite[Theorem A]{jolissaint2}:

\begin{fact}\label{fact2}
Suppose that $G$ has property RD.  Then $H^\infty(G)$ is stable under holomorphic functional calculus in $\Cr$:  if $z\in H^\infty(G)$, $U\subseteq \mathbb{C}$ is an open neighborhood of $\spec(z)$, and $f:U\to \mathbb{C}$ is holomorphic, then $f(z)\in H^\infty(G)$.
\end{fact}

We can now prove our main result:

\begin{theorem}\label{thm:main}
Let $G$ be finitely generated with solvable word problem and property RD.  Then the norm on $\Cr$ is computable.  Moreover, the algorithm for computing the norm depends only on a computable presentation for $G$, an algorithm for solving the word problem for $G$, and a pair of RD constants for $G$.
\end{theorem}

\begin{proof}
As mentioned in the introduction, one can always compute a sequence of lower bounds converging to the norm on $\|\cdot\|_r$, so we only need to explain the existence of computable upper bounds.

Given $x\in \mathbb{Q}(i)[G]$, we seek to effectively enumerate the rational upper bounds of $\|x\|_r$.  Let $(C,s)$ be RD constants for $G$.  Given $q\in\Q^{>0}$, search over pairs $(t,w)$ with $t\in\Q$, $0<t<q$, and $w\in\Qi[G]$, computing
\[
r\ :=\ t^2-x^*x-w^*w\ \in\ \Qi[G],
\]
and accepting if
\begin{equation}\label{eq:cert}
t^2+C\,\|r\|_{H^s}\ <\ q^2 .
\end{equation}
Note that \eqref{eq:cert} is equivalent to the inequality $C^2\|r\|_{H_s}^2<(q^2-t^2)^2$, which is a comparison between two rational numbers that is, by Lemma \ref{sobolevcomp}, a decidable test.  

We claim that this algorithm enumerates all rational upper bounds of $\|x\|_r$.  To see this, first suppose that \eqref{eq:cert} holds. In $\Cr$, we have
\[
x^*x\ =\ t^2-w^*w-r\ \le\ t^2+\|r\|_r.
\]
By \eqref{eq:rd}, $\|r\|_r\le C\|r\|_{H^s}$, so $\|x\|_r^2=\|x^*x\|_r\le t^2+C\|r\|_{H^s}<q^2$, giving $\|x\|_r<q$.

Conversely, suppose $\|x\|_r<q$ and choose $t\in\Q$ with $\|x\|_r<t<q$. Set $z:=t^2-x^*x\in\Qi[G]$, so
\[
\spec_{\Cr}(z)\ \subseteq\ \big[\,t^2-\|x\|_r^2,\ t^2\,\big]\ \subseteq\ (0,\infty).
\]
Consequently, $z$ is positive and invertible in $\Cr$.  Set $v:=z^{1/2}\in \Cr$.  By Fact \ref{fact2} applied to some open $U\subseteq \mathbb{C}\setminus (-\infty,0]$ containing $\spec(z)$ and$f:U\to \mathbb{C}$ the principal branch of the square root function, we in fact have that $v\in H^\infty(G)$.

Set $B_n:=\{g:\ell(g)\le n\}$.  
We then have
\begin{alignat}{2}
\|v-v1_{B_n}\|_{H^{s}}^2&=\sum_{\ell(g)>n}|v(g)|^2(1+\ell(g))^{2s}\notag \\ \notag
    &\le(1+n)^{-2}\!\!\sum_{\ell(g)>n}\!\!|v(g)|^2(1+\ell(g))^{2s+2}\\ \notag
    &\le(1+n)^{-2}\|v\|_{H^{s+1}}^2.\notag
\end{alignat}
Here, we used the fact that $(1+\ell(g))^{2s}\le(1+n)^{-2}(1+\ell(g))^{2s+2}$ for $\ell(g)>n$.  As a result, we have that $\|v-v1_{B_n}\|_{H^{s}}\to 0$ as $n\to \infty$.

For $h$ supported on the finite set $B_n$, one has
\begin{equation}
\|h\|_{H^{s}}\ \le\ (1+n)^{s}\,\|h\|_{\ell^2},\notag
\end{equation}
so replacing the coefficients of $v1_{B_n}$ by elements of $\Qi$ within $\eta(1+n)^{-s}|B_n|^{-1/2}$ of them changes $\|\cdot\|_{H^{s}}$ by at most $\eta$.  Combining these two facts, we see that $v$ can be approximated arbitrarily well in the $\|\cdot\|_{H^{s}}$ norm by elements of $\Qi[G]$.

Finally, for $w\in \Qi[G]$, we have:
\begin{alignat}{2}
\|z-w^*w\|_{H^s}&=\|v^*v-w^*w\|_{H^s} \notag \\ \notag
    &\le\|v^*(v-w)\|_{H^s}+\|(v-w)^*w\|_{H^s}\\ \notag
&\le \big(\|v\|_{H^{s}}+\|w\|_{H^{s}}\big)\|v-w\|_{H^{s}}.\notag
\end{alignat}
Here, we used Fact \ref{fact1} and the easy fact that $\|h^*\|_{H^2}=\|h\|_{H^s}$ for all $h\in \mathbb{C}[G]$.  Since $t^2<q^2$, by the previous paragraph we may find $w\in \Qi[G]$ so that $t^2+C\|z-w^*w\|_{H^s}<q^2$, whence \eqref{eq:cert} is satisfied and the search terminates.
\end{proof}

By a result of Lafforgue \cite{lafforgue}, discrete cocompact lattices in $\operatorname{SL}_3(\mathbb{R})$ have property RD, whence our main theorem covers such groups.  On the other hand, by a result of Jolissaint \cite[Theorem B]{jolissaint}, $\operatorname{SL}_3(\mathbb{Z})$ does not have property RD.  It would therefore be interesting to settle the following:

\begin{question}
Is the norm on $C^*_r(\operatorname{SL}_3(\mathbb{Z})$ computable?
\end{question}

The analogous question for the universal group \cstar-algebra $C^*(\operatorname{SL}_3(\mathbb{Z})$ was posed by Fritz, Netzer, and Thom \cite[Question 2.6]{canyou} and still remains open.

\section{A noncomputable presentation of $C^*_r(G)$}

In this section, we prove the following theorem:

\begin{theorem}
There is a finitely presented group $G$ with solvable word problem for which the norm on $C^*_r(G)$ is not computable.
\end{theorem}

\begin{proof}
We first construct a countably generated group $H$ with solvable word problem such that the norm on $C^*_r(H)$ is not computable.  We later explain how to achieve a finitely presented example.

The idea behind the construction of $H$ is very simple:  we let $H=\bigoplus H_n$, where $(H_n)_{n\in \N}$ is a sequence of 2-generated groups $H_n=\langle a_n,b_n\rangle$ satisfying the following hypotheses:
\begin{enumerate}
    \item There is a single algorithm that solves the word problem in each $H_n$.
    \item If $K$ denotes the halting set, then $H_n$ is a finite group if $n\in K$ while $H_n$ is freely generated by $a_n$ and $b_n$ if $n\notin K$.
\end{enumerate}

Assume momentarily that these groups have been constructed satisfying the above hypotheses.  For each $n$, consider $$m_n = (1/4)(a_n+a_n^{-1}+b_n+b_n^{-1}) \in \mathbb{Q}[H_n]\subseteq \mathbb{Q}[H].$$  Note that $\|m_n\|_r$ is the same as whether calculated in $\mathbb{C}[H]$ or $\mathbb{C}[H_n]$ since the Hilbert space $\ell^2(H)$ decomposes as a direct sum $\ell^2(H)=\bigoplus_{H_ng\in H/H_n} \ell^2(H_ng)$ and the action of an element of $H_n$ on $\ell^2(H)$ is equivalent to a direct sum of copies of the action of the element on $H_n$.

Now if $n\in K$, then $\|m_n\|=1$ by Kesten's criterion of amenability, while if $n\notin K$, then $\|m_n\|_r=\sqrt{3}/2$ by a well-known computation.  It follows that if $\|\cdot\|_r$ were computable, then one could decide the halting problem, a contradiction.

We now explain how to construct the groups $H_n$.  Fix an enumeration $(T_n)_{n\in \N}$ of Turing machines. For each $t\in \N$, effectively choose a finite quotient
$q_t : \mathbb{F}_2 \to Q_t$
that is injective on the ball of radius $t$; this is possible since $\mathbb{F}_2$ is residually finite.  Note also that these surjections can be found effectively by searching through finite groups. We let $a_t,b_t\in Q_t$ denote the images of the free generators of $\mathbb{F}_2$ under $q_t$.

Now define
$H_n:=Q_t$ if $T_n$ first halts at time $t$, while set  $H_n:=\mathbb{F}_2$ if $T_n$ never halts.  By definition, hypothesis (2) above holds.  The following algorithm verifies criterion (1):  given a word $w$ of length $L$ in $a_n,b_n$, do the following:

\begin{enumerate}
    \item[(i)] Freely reduce $w$.
\item[(ii)] Run $T_n$ for $L$ steps.
\item[(iii)] If the run halts at some $t \leq L$, evaluate $w$ in $Q_t$ and check if it is the identity.
\item[(iv)] If the run has not halted by time $L$, then $w$ is trivial exactly when its free reduction is the empty sequence. 
\end{enumerate}

The last step is valid because if $T_n$ eventually halts at a later time $t>L$, then since the map $q_t$ is injective on the $L$-ball, no nontrivial word of length $L$ is mapped to the identity under $q_t$.

It remains to construct a finitely presented group satisfying the conclusion.  By results of A. Ould Houcine \cite[Theorems I and II]{ouldhoucine}, the fact that $Z(H)$ is c.e. and has decidable membership problem (which is proved in a manner analogous to the argument for establishing solvability of the word problem for $H$) implies that $H$ can be embedded in a finitely presented group $G$ with solvable word problem.  Moreover, as the arguments in \cite{ouldhoucine} show, the embedding $H\hookrightarrow G$ is effective and so leads to a computable embedding $C^*_r(H)\hookrightarrow C^*_r(G)$.  As a result, $C^*_r(G)$ does not have computable norm, for otherwise $C^*_r(H)$ would have computable norm.
\end{proof}

\end{document}